\documentclass{amsart}

\usepackage{amsmath,amsfonts,amssymb,amsthm}
\newcommand{\abs}[1]{|#1|}
\newcommand{\floor}[1]{\left\lfloor#1\right\rfloor}
\newcommand{\len}{\ell}

\DeclareMathOperator{\diam}{diam}

\newcommand{\wt}{\widetilde}
\newcommand{\ol}{\overline}

\newtheorem{thm}{Theorem}[section]
\newtheorem{lem}[thm]{Lemma}

\begin{document}

\title{Expanders are order diameter non-hyperbolic}
\author{Anton Malyshev}
\begin{abstract}
  We show that expander graphs must have Gromov-hyperbolicity at least proportional to their diameter, with a constant of proportionality depending only on the expansion constant and maximal degree. In other words, expanders contain geodesic triangles which are $\Omega(\diam\Gamma)$-thick.
\end{abstract}

\maketitle

\section{Introduction}

An expander graph is a finite graph of bounded degree in which all sets of vertices have large boundary (see e.g., \cite{expander_survey}). A hyperbolic graph is a graph which behaves like a tree on large scales (see e.g., \cite{hyperbolic_survey}).

It is shown in~\cite{enh} that sufficiently large graphs cannot be both expanders and hyperbolic. More precisely, expander graphs $\Gamma$ satisfy
\begin{align*}
  \delta_\Gamma = \Omega\left( \frac{\log\abs\Gamma}{\log\log\abs\Gamma} \right),
\end{align*}
where $\delta_\Gamma$ is the Gromov-hyperbolicity constant of the graph $\Gamma$, i.e., the minimal $\delta$ such that all geodesic triangles in $\Gamma$ are \mbox{$\delta$-thin}.

We modify the argument slightly to improve this result to 
\begin{align*}
  \delta_\Gamma = \Omega\left( \log\abs\Gamma \right).
\end{align*}
For a more precise statement, see Theorem~\ref{thm:expanders_not_hyperbolic}.

One new consequence of this result is that a family of expanders cannot have a tree as an asymptotic cone.  

A similar result is known for vertex-transitive graphs, that is, if $\Gamma$ is a finite vertex-transitive graph, then
\begin{align*}
  \delta_\Gamma = \Omega(\diam G).
\end{align*}
See, e.g., \cite{ftrans}. This fact can be used to give an elementary argument that hyperbolic groups only have finitely many finite subgroups, up to conjugacy \cite{brady}. 

There is also a similar result for random graphs. A random \mbox{$d$-regular} graph with~$n$ vertices is asymptotically almost surely an expander, and expanders share many properties with random graphs. In \cite{geo_random}, it is shown that random graphs have large almost-geodesic cycles. (It is open whether the same holds for true geodesic cycles.) Another consequence of their argument is that a random \mbox{$d$-regular} graph with~$n$ vertices, $\Gamma$, asymptotically almost surely satisfies
\begin{align*}
  \delta_\Gamma = \frac12 \diam(\Gamma) - O(\log\log\abs{\Gamma}) = \frac12 \diam(\Gamma) (1 - o(1)).
\end{align*}
The constant $1/2$ is the best possible, since for any graph $\Gamma$, the hyperbolicity constant $\delta_\Gamma$ is bounded above by $\diam(\Gamma)/2$.

While an expander cannot be hyperbolic, it is possible that it has a high proportion of thin geodesic triangles. In \cite{manythin}, Li and Tucci construct a family of expanders in which a positive proportion of geodesics pass through a particular vertex. It follows that a positive proportion of geodesic triangles are tripods, i.e., \mbox{$0$-thin}. Indeed, their argument constructs expander families in which this proportion is arbitrarily close to~1, at the cost of worse expansion. This leaves the question: is there a fixed~$\delta$ and an expander family $(\Gamma_i)_i$, in which geodesic triangles are~\mbox{$\delta$-thin} asymptotically almost surely, i.e., in which the probability that a randomly chosen geodesic triangle in $\Gamma_i$ is~\mbox{$\delta$-thin} tends to 1 as $\abs{\Gamma_i} \to \infty$?

\section{Conventions}

A \emph{graph} is a finite undirected graph, which may have loops and repeated edges. We abuse notation and use $\Gamma$ to denote both a graph and its set of vertices. A graph~$\Gamma$ is an \emph{\mbox{$h$-expander}} if for every set $S \subseteq \Gamma$ with~\mbox{$\abs{S} \leq \abs{\Gamma}/2$}, there are at least~$h \abs{S}$ vertices in $\Gamma \setminus S$ adjacent to $S$. A graph has \emph{valency bounded by $d$} if every vertex in the graph has degree at most~$d$. We let $B_{\Gamma, p}(r)$ denote the ball in $\Gamma$ with center~$p$ and radius~$r$.

A \emph{\mbox{$\delta$-hyperbolic}} space is a geodesic metric space in which all geodesic triangles are \emph{\mbox{$\delta$-thin}}. That is, each side of a triangle is contained in \mbox{a $\delta$-neighborhood} of the other two sides. If $X$ is a geodesic metric space, and $p,q \in X$, we let $\ol{pq}$ denote a geodesic segment between $p$ and $q$. This is a slight abuse of notation because such a path need not be unique, but this does not cause any issues in our arguments.

A graph can be thought of as a geodesic space in which each edge is a segment of length $1$. A finite graph $\Gamma$ is always \mbox{$\delta$-geodesic} for some $\delta$, e.g.,~$\delta$ equal to the diameter of the graph. We write $\delta_\Gamma$ for the minimal such $\delta$.

\section{Proof of Main Result}

In~\cite{enh}, the argument proceeds by showing that, in an expander, removing a ball does not significantly affect expansion, so the diameter of the graph does not increase much. On the other hand, removing a ball in a hyperbolic graph increases distances by an amount exponential in the radius of the ball.

More precisely, let $D$ be the diameter of an expander $\Gamma$, i.e., the maximal distance between two points $p,q \in \Gamma$. Note that  $D = \Theta(\log\abs\Gamma)$. Removing a ball of radius~$\Theta(D)$ between $p$ and $q$ shrinks boundaries of balls centered at $p$ and $q$ by an additive amount $\exp(\Theta(D))$, which increases $d(p,q)$ by at most a constant factor. However, in a \mbox{$\delta$-hyperbolic} graph, removing such a ball must increase $d(p,q)$ by~$\exp(\Omega(D/\delta))$. Hence,
\begin{align*}
  D + \exp(\Omega(D/\delta)) = O(D),
\end{align*}
and therefore
\begin{align*}
  \delta = \Omega(D/\log D).
\end{align*}

The $\log D$ factor arises because we compare a multiplicative increase of $d(p,q)$ in an expander to an additive increase of $d(p,q)$ in a hyperbolic graph. We eliminate this factor by using a~\mbox{$\Theta(D)$-neighborhood} of a~\mbox{$\Theta(D)$-length} path between $p$ and $q$, instead of a ball of radius $\Theta(D)$. With an appropriate choice of constants, removing such a cylinder from an expander still changes $d(p,q)$ by at most a constant factor. However, in a \mbox{$\delta$-hyperbolic} graph, removing such a cylinder increases $d(p,q)$ by a \emph{multiplicative} factor of $\exp(\Omega(D/\delta))$. So we obtain
\begin{align*}
  D\exp(\Omega(D/\delta)) = O(D),
\end{align*}
and therefore
\begin{align*}
  \delta = \Omega(D).
\end{align*}

The key lemma, then, is the increase of distances in a hyperbolic graph caused by removing a cylinder. We prove this lemma using some elementary facts about hyperbolic spaces, whose statements and proofs we defer to section~\ref{sec:hyperbolic_lemmas}.

\begin{lem} \label{lem:longpath}
Let $X$ be a \mbox{$\delta$-hyperbolic} space, and let $p,q \in X$. Define~$D = d(p,q)$, and choose $p', q'$ on the geodesic segment $\ol{pq}$ so that the points~$p$,~$p'$,~$q'$,~$q$ are evenly spaced, that is,
\begin{align*}
  d(p,p') = d(p',q') = d(q',q) = D/3.
\end{align*}
Let $\gamma$ be a path in $X$ which does not intersect the \mbox{$R$-neighborhood} of the geodesic segment $\ol{p'q'}$. Then
  \begin{align*}
    \len(\gamma) \geq D B 2^{R/\delta}
  \end{align*}
where $B$ is a universal constant, and $R/\delta$, $D/\delta$ are taken sufficiently large.
\end{lem}
\begin{proof}[Proof of Lemma~\ref{lem:longpath}]
  Consider the set of all pairs $(x,r)$ with $x$ on the path $\gamma$ and~$r$ on the geodesic $\ol{pq}$, for which
  \begin{eqnarray*}
    d(x, r) = \min_{s \in \ol{pq}} d(x, s),
  \end{eqnarray*}
i.e., $r$ is the nearest point of $\ol{pq}$ to $x$. By Lemma~\ref{lem:skeleton}, there is a constant~$C$ such that every interval of length $C\delta$ in $\ol{pq}$ contains some such $r$. If we let
  \begin{align*}
    N = \floor{ \frac{D/3}{C\delta + K_0\delta} },
  \end{align*}
  we can find $N$ intervals of length $C\delta$ in $\ol{p'q'}$, spaced more than $K_0\delta$ apart, and pick such a point $r$ in each interval. So, there is a collection of $N$ such pairs $(r_i,x_i)$ for which the $r_i$ are more than $K_0\delta$ apart from each other. We may rearrange these pairs so that the $x_i$ occur in order of appearance on $\gamma$.

  Since the $r_i$ are sufficiently far apart, Lemma \ref{lem:segment} guarantees that each geodesic~$\ol{x_ix_{i+1}}$ passes within $K_1\delta$ of $\ol{p'q'}$. However, the path $\gamma$ does not pass within~$R$ of~$\ol{p'q'}$. Hence, some point on $\ol{x_ix_{i+1}}$ is at distance at least~\mbox{$R - K_1\delta$} from $\gamma$. So, by Lemma \ref{lem:balldetour}, the segment of $\gamma$ between $x_i$ and $x_{i+1}$ has length at least $\delta 2^{R/\delta-K_1}$.
  
There are $N - 1$ such segments, so the total length of $\gamma$ is at least
  \begin{align*}
    \left( N - 1 \right) \delta 2^{R/\delta - K_1}
    \geq D \left( \frac{1}{3(C+ K_0)} - \frac{2}{D/\delta} \right) 2^{R/\delta}.
  \end{align*}
  Taking, e.g., $B = (4(C+K_0))^{-1}$ gives us
  \begin{align*}
    \len(\gamma) \geq D B 2^{R/\delta},
  \end{align*}
  for sufficiently large $D/\delta$.
\end{proof}

We can now prove our result by arguing that in an expander, removing such a cylinder does not increase distances too much.

\begin{thm} \label{thm:expanders_not_hyperbolic}
  For any positive real number~$h$ and any positive integer~$d$, there is a constant $C_{d,h}$ such that for any \mbox{$h$-expander} $\Gamma$ which has valency bounded by $d$, we have
  \begin{align*}
    \delta_\Gamma > C_{d,h} \log \abs{\Gamma}.
  \end{align*}
\end{thm}

\begin{proof}%[Proof of Theorem~\ref{thm:expanders_not_hyperbolic}]
  Choose $\alpha$ small enough that 
  \begin{align*}
    (1+h)^{1/3-\alpha} < d^{\alpha} 
  \end{align*}

  Pick $p,q \in \Gamma$ such that $d(p,q)$ is divisible by $3$ and as large as possible, and let~$D = d(p,q)$. Let $p'$, $q'$ be points on the geodesic segment from $p$ to $q$ such that~$p$,~$p'$,~$q'$,~$q$ are equally spaced, i.e.\ 
  \begin{align*}
    d(p,p') = d(p',q') = d(q', q) = D/3.
  \end{align*}
  Let $C$ be the \mbox{$(\alpha D)$-neighborhood} of the geodesic segment $\ol{p'q'}$, and let $\wt \Gamma = \Gamma \setminus C$ be the graph $\Gamma$ with this cylinder removed.

  We consider the growth of the balls $B_{\wt\Gamma,p}(r)$ in $\wt\Gamma$ centered at $p$. Roughly speaking, by the time these balls reach the cylinder $C$, they are large enough that their growth is not significantly slowed down by the cylinder.

  As long as $\abs{B_{\wt\Gamma,p}(r-1)} < \abs{\Gamma}/2$, the ball $B_{\wt\Gamma,p}(r-1)$ has a large boundary in $\Gamma$. When~$r < D/3 - \alpha D$, the ball has not yet reached the cylinder $C$, so this boundary lies entirely in $\wt \Gamma$, and we have
  \begin{align*}
    \abs{B_{\wt\Gamma,p}(r)} \geq \abs{B_{\wt\Gamma,p}(r-1)}(1 + h)
  \end{align*}
  and in particular, by the time the ball reaches the cylinder, we have
  \begin{align*}
    \abs{B_{\wt\Gamma,p}( D/3 - \alpha D - 1)} \geq (1+h)^{ D/3 - \alpha D-1}.
  \end{align*}
  For larger $r$, the growth of the ball in $\wt\Gamma$ is reduced by at most the size of the cylinder, so we have
  \begin{align*}
    \abs{B_{\wt\Gamma,p}(r)} \geq \abs{B_{\wt\Gamma,p}(r-1)}(1 + h) - \abs{C}.
  \end{align*}
  Hence, as long as we still have $\abs{B_{\wt\Gamma,p}(r-1)} < \abs \Gamma / 2$, and the cylinder is small enough compared to the ball, say,
  \begin{align}\label{eqn:smcyn}
    \abs{C} < \abs{B_{\wt\Gamma,p}(r-1)} h/2,
  \end{align}
  the balls continue to grow:
  \begin{align}\label{eqn:bg}
    \abs{B_{\wt\Gamma,p}(r)} \geq \abs{B_{\wt\Gamma,p}(r-1)}(1 + h/2),
  \end{align}

  To see the necessary inequality (\ref{eqn:smcyn}), observe that the cylinder cannot be too large,
  \begin{align*}
    \abs{C} \leq (D/3+1) d^{\alpha D},
  \end{align*}
  and as we noted earlier, for $r \geq D/3 - \alpha D$, the ball has size at least
  \begin{align*}
    \abs{B_{\wt\Gamma,p}(r)} \geq (1 + h)^{D/3-\alpha D - 1}.
  \end{align*}
  Hence,
  \begin{align*}
    \abs{C} / \abs{B_{\wt\Gamma,p}(r-1)}
    &\leq
    (D/3+1) d^{\alpha D}  (1 + h)^{-D/3+\alpha D} 
    \\
    &=
    (D/3+1) \left(d^{\alpha}  (1 + h)^{-1/3+\alpha}\right)^D
  \end{align*}
  By our choice of $\alpha$, this goes to $0$ as $D \to \infty$. So when $D$ is sufficiently large, we have 
  \begin{align*}
    \abs{C} / \abs{B_{\wt\Gamma,p}(r-1)} < h/2,
  \end{align*}
  as desired.

  From equation (\ref{eqn:bg}), we have that
  \begin{align*}
    \abs{B_{\wt\Gamma,p}(r)} \geq \min\left( (1+h/2)^r, \abs \Gamma / 2 \right),
  \end{align*}
  and hence if $r \geq \log \abs \Gamma / \log (1+h/2)$, we have
  \begin{align*}
    \abs{B_{\wt\Gamma,p}(r)} \geq \abs \Gamma / 2.
  \end{align*}
  The same holds for $\abs{B_{\wt\Gamma,q}(r)}$, so $B_{\wt\Gamma,p}(r)$ and $B_{\wt\Gamma,q}(r)$ intersect. Hence,
  \begin{align*}
    d_{\wt \Gamma}(p,q) \leq 2 \log \abs \Gamma / \log(1+h/2)
  \end{align*}

  However, by Lemma~\ref{lem:longpath}, if $\Gamma$ is \mbox{$\delta$-hyperbolic} we have two possibilities. Either~$D/\delta$ is not sufficiently large, in which case and $\delta$ is bounded below by a constant times~$D$, or it is sufficiently large, and the distance between~$p$~and~$q$ increases to
  \begin{align*}
    d_{\wt \Gamma}(p,q) \geq  D B 2^{\alpha D/\delta}.
  \end{align*}
  Hence,
  \begin{align*}
    D B 2^{\alpha D/\delta} \leq 
    \log \abs \Gamma / \log(1+h/2).
  \end{align*}
  And therefore,
  \begin{align*}
    \frac D \delta \leq 
    \frac{1}{\alpha \log 2} \log \left( \frac{1}{B \log(1+h/2)} \frac{\log\abs\Gamma}{D} \right)
  \end{align*}

  However, $D$ is bounded below by $\log\abs\Gamma / \log d$, so the right hand side is is bounded above by a constant. Hence, there are constants $K, K' > 0$ such that when $\abs\Gamma$ is sufficiently large we have
  \begin{align*}
    \delta \geq K D \geq K' \log\abs\Gamma.
  \end{align*}

\end{proof}

\section{Elementary Lemmas about Hyperbolic Spaces} \label{sec:hyperbolic_lemmas}

We now state and prove the necessary basic lemmas about \mbox{$\delta$-hyperbolic} spaces. These results appear in the literature, but we include proofs for completeness.

\begin{lem}\label{lem:isoc}
  Let $X$ be a \mbox{$\delta$-hyperbolic} space, and let $x,y,z \in X$ such that
  \begin{align*}
    d(z,x) = d(z,y) = \min_{w \in \ol{xy}} d(z,w),
  \end{align*}
  that is, $x$ and $y$ are both points of minimal distance to $z$ on a geodesic segment $\ol{xy}$ between them.
  Then $d(x,y) < K\delta$, where~$K$ is some universal constant.
\end{lem}
\begin{proof}
  Since every point on $\ol{xy}$ is within $\delta$ of either $\ol{zx}$ or $\ol{zy}$, by continuity, there must be a point $w$ on $\ol{xy}$ within $\delta$ of both $\ol{zx}$ and $\ol{zy}$.
  \begin{align*}
    d(z,x) \geq d(z,w) + d(w,x) - 2\delta \geq d(z,x) + d(w,x) - 2\delta.
  \end{align*}
  Hence $d(w,x) \leq 2\delta$. Similarly, $d(w,y) \leq 2\delta$, so $d(x,y) \leq 4\delta$.
\end{proof}

\begin{lem} \label{lem:skeleton}
There is a universal constant $C$ such that the following is true.

  Let $X$ be a \mbox{$\delta$-hyperbolic} space, and let $p,q$ be points in $X$. Let $\beta$ be some path from $p$ to $q$. For any interval of length $C\delta$ in the geodesic segment $\ol{pq}$, there is a pair $(x, r)$ such that $r$ is in that interval, $x$ is on the path $\beta$, and $r$ is the closest point on $\ol{pq}$ to $x$.
\end{lem}
\begin{proof}
  We write $\alpha:[0,L] \to X$ for the parametrization of~$\ol{pq}$ by arclength. We reparametrize~$\beta$ if necessary so we have~$\beta:[0,1] \to X$, with
  \begin{align*}
    \alpha(0) = \beta(0) = p \quad \text{and} \quad \alpha(L) = \beta(1) = q.
  \end{align*}

  Consider the set of pairs 
  \begin{align*}
    S = \left\{ (t,s) \in [0,1] \times [0,L] \  : \  d(\beta(t), \alpha(s)) = \min_{s' \in [0,L]} d(\beta(t),\alpha(s')) \right\}
  \end{align*}
  The set $S$ is closed, hence compact. By Lemma~\ref{lem:isoc}, each set
  \begin{align*}
    S_t = \{ s : (t,s) \in S\}
  \end{align*}
  has diameter at most $K\delta$. Given any $s \in [0,L]$, the sets
  \begin{align*}
    \{ t \in [0,1] : S_t \cap [0,s] \neq \emptyset \} \quad \text{ and } \quad \{ t \in [0,1] : S_t \cap [s,L] \neq \emptyset \}
  \end{align*}
  are closed, nonempty, and cover $[0,1]$. So some $t$ belongs to both, that is, there is a value~$t$ such that $S_{t}$ intersects both $[0,s]$ and $[s,L]$. Since $S_{t}$ has diameter at most $K \delta$, it contains some element of~$[s - K\delta, s + K\delta]$. Thus, some element of~$[s - K\delta, s+K\delta]$ belongs to $S_t$.

  So, every closed interval of length $2K\delta$ in $[0,L]$ intersects some $S_t$, as desired.  
\end{proof}

\begin{lem} \label{lem:balldetour}
  Let $X$ be a \mbox{$\delta$-hyperbolic} space and let $\gamma$ be a path from $p$ to $q$ in $X$. If some point $r$ on $\ol{pq}$ is at distance $R$ from $\gamma$, then $\len(\gamma) \geq \delta 2^{R/\delta}$.
\end{lem}
\begin{proof}
  We prove this by induction on $\floor{R/\delta}$. By assumption, some point on $\ol{pq}$ is at distance $R$ from both $p$ and $q$, so $\len(\gamma) \geq d(p,q) \geq 2R$. In particular, if $\floor{R/\delta} = 0$, then $\len(\gamma) \geq 2R \geq \delta 2^{R/\delta}$.

  For the inductive step, let $x$ be the point on $\gamma$ which is halfway between~$p$~and~$q$. Then~$r \in \ol{pq}$ is within $\delta$ of either $\ol{px}$ or $\ol{xq}$. We may suppose without loss of generality that it is $\ol{px}$. Then some point on $\ol{px}$ has distance at least $R - \delta$ from~$\gamma$. It follows that~$\len(\gamma)/2 \geq \delta 2^{(R-\delta)/\delta}$, and therefore $\len(\gamma) \geq \delta 2^{R/\delta}$.
\end{proof}

\begin{lem} \label{lem:segment}
  There are universal constants $K_0$ and $K_1$ such that the following is true.

  Let $X$ be a \mbox{$\delta$-hyperbolic} space, and let $p,q,x,y$ be points in~$X$, where we suppose~\mbox{$d(p,q) > K_0 \delta$}. If the nearest points on $\ol{pq}$ to $x$ and $y$ are~$p$~and~$q$, respectively, then $\ol{xy}$ passes within $K_1 \delta$ of $\ol{pq}$.
\end{lem}
\begin{proof}
  We take $K_0 = 12$ and $K_1 = 2$.

  The geodesic $\ol{xq}$ stays within a \mbox{$\delta$-neighborhood} of $\ol{xp} \cup \ol{pq}$, and the geodesic $\ol{xy}$ stays within a \mbox{$\delta$-neighborhood} of $\ol{xq} \cup \ol{qy}$, hence $\ol{xy}$ stays within a \mbox{$2\delta$-neighborhood} of $\ol{xp} \cup \ol{pq} \cup \ol{qy}$. If $\ol{xy}$ does not pass within~$K_1 \delta = 2\delta$ of $\ol{pq}$, then there must be a point $z$ on $\ol{xy}$ which is within $2\delta$ of both $\ol{xp}$ and $\ol{qy}$.

  In particular, $d(x,z) \leq d(x,p) - d(z,p) + 4\delta$. There is a point~$r$ on~$\ol{pq}$ which is within $\delta$ of both $\ol{zp}$ and $\ol{zq}$, so $d(z,r) \leq d(z,p) - d(r,p) + 2\delta$. Hence,
  \begin{align*}
    d(x,r)
    &\leq d(x,z) + d(z,r)
    \\
    &\leq d(x,p) - d(z,p) + 4\delta + d(z,p) - d(r,p) + 2\delta
    \\
    &= d(x,p) - d(r,p) + 6\delta
    \\
    &\leq d(x,r) - d(r,p) + 6\delta.
  \end{align*}
  Hence, $d(r,p) \leq 6\delta$. Similarly, $d(r,q) \leq 6\delta$ so $d(p,q) \leq 12\delta = K_0 \delta$. So, if~$\ol{xy}$ does not pass within~$K_1\delta$ of~$\ol{pq}$ then $d(p,q) \leq K_0 \delta$, as desired.
\end{proof}

\end{document}